\documentclass{article}
\usepackage[utf8]{inputenc}
\usepackage{algorithm2e}
\usepackage{mathtools}
\usepackage{amsthm}
\usepackage{cite}
\usepackage{comment}
\usepackage{authblk}
\usepackage{todonotes}
\usepackage{hyperref}
\usepackage[toc,page]{appendix}
\usepackage[dvipsnames]{xcolor}
\usepackage{amssymb}
\usepackage{newpxtext}

\newtheorem{lemma}{Lemma}
\newtheorem{theorem}{Theorem}
\newtheorem{proposition}{Proposition}
\newtheorem{corollary}{Corollary}

\DeclarePairedDelimiter\floor{\lfloor}{\rfloor}

\RequirePackage{fix-cm}
\usepackage{comment}
\usepackage{amssymb}
\usepackage{hyperref}
\usepackage{longtable}
\usepackage{fontawesome5}

\begin{document}
\title{New bounds on the Graham-Pollak theorem for hypergraphs}
\author{Anand Babu\thanks{Department of Computer Science {\&} Engineering,
National Institute of Technology Calicut, Kozhikode, India.
Corresponding author:
\href{mailto:anandbabu@nitc.ac.in}{\texttt{anandbabu@nitc.ac.in}}.
ORCID:
\href{https://orcid.org/0000-0002-9904-0488}{0000-0002-9904-0488}}}
\date{}
\maketitle
\begin{abstract}
For a fixed $r$, let $f_r(n)$ denote the minimum number of complete $r$-partite $r$-uniform hypergraphs whose edge sets partition the complete $r$-uniform hypergraph on $n$ vertices. The Graham-Pollak theorem asserts that $f_2(n)=n-1$. Let $c_r$ be the smallest constant such that $f_r(n)\le c_r(1+o(1))\binom{n}{\lfloor r/2\rfloor}$. We give two constructions that improve upper bounds for $f_r(n)$. The first partitions $E(K_4)\times E(K_{17})$ into $44$ products of edge sets of complete bipartite graphs and yields $c_4\le11/12$, improving the previous bound of $c_4\le14/15$. The second gives an asymptotically improved exact cover for odd $r$. Combining these constructions, we prove that $c_r<1$ for every odd integer $r\ge65$, which improves upon the previous computer-assisted result of $113$. We further obtain the improved asymptotic bound
$c_r\le \frac{r}{15}\left(\frac{11}{12}\right)^{r/4}+o(1)$ for general $r$. These results narrow the gap towards resolving the major open problem of determining whether $c_5<1$.
\end{abstract}

biclique partition.


\section{Introduction}
The \emph{biclique partition number} $bp(G)$ of a graph $G$ is the minimum number of complete bipartite graphs (bicliques) whose edge sets partition $E(G)$. An $r$-uniform hypergraph $H$ (or simply an $r$-graph) is \emph{$r$-partite} if its vertex set admits a partition $V(H)=V_1\cup \cdots \cup V_r$ such that every edge contains exactly one vertex from each part. Equivalently, $E(H)\subseteq V_1\times\cdots\times V_r$. The celebrated Graham-Pollak theorem~\cite{graham1971addressing, graham1972embedding} states that $bp(K_n) = n-1$. The original proof by Graham and Pollak uses Sylvester's law of inertia \cite{graham1972embedding}. Other proofs of the same were found by Tverberg~\cite{tverberg1982decomposition}, Peck~\cite{peck1984new} and Vishwanathan~\cite{vishwanathan2008polynomial} using linear algebraic methods. Vishwanathan later gave a combinatorial proof~\cite{vishwanathan2013counting}.

Aharoni and Linial \cite{alon1986decomposition} posed the natural extension of this problem to determine the minimum size of the family of complete $r$-partite $r$-graphs whose edge sets partition the edge set of the complete $r$-uniform hypergraph, for $r>2$. We denote by $f_r(n)$ the minimum size of such a family that partitions the edge set of the complete $r$-uniform hypergraph on $n$ vertices, $K_n^r$. For $r=2$, $f_r(n)$ denotes the biclique partition number of $K_n$.

For $r=3$, Alon~\cite{alon1986decomposition} showed that $f_3(n)=n-2$. For general $r$, Alon~\cite{alon1986decomposition} established an upper bound that was later improved by Cioab\u{a}, K\"undgen, and Verstra\"ete~\cite{cioabua2009decompositions}. Alon~\cite{alon1986decomposition} also established the corresponding linear-algebraic lower bound. Thus,
${\frac{2}{\binom{2\lfloor r/2 \rfloor}{\lfloor r/2 \rfloor}}(1+o(1))\binom{n}{\lfloor r/2 \rfloor}
\le f_r(n)
\le (1+o(1))\binom{n}{\lfloor r/2 \rfloor}}$. Later Leader, Mili{\'c}evi{\'c}, and Tan \cite{leader2017decomposing} proved that ${f_4(n)\le \left(\frac{14}{15}\right)(1+o(1))\binom{n}{2}}$ and as an immediate consequence, they obtained $f_r(n)\le
\left(\frac{14}{15}\right)(1+o(1))
\binom{n}{r/2}$
for even $r$. Let $c_r$ denote the smallest constant $c$ such that
$f_r(n)\le
c(1+o(1))
\binom{n}{\lfloor r/2\rfloor}$. Leader and Tan \cite{leader2018improved} subsequently proved that, for every $r\ge4$, $c_r\le
\frac{r}{2}
\left(\frac{14}{15}\right)^{r/4}
+o(1)$,
which implies that $c_r<1$ for all odd $r\ge295$. Babu and Vishwanathan \cite{babu2019bounds} improved this to $113$. Since Alon's result gives $c_3=1$, $c_5$ is the first unresolved case. In this note, we prove $c_r<1$ for every odd $r\ge65$, improving the previous bound of $r\ge113$. This narrows the gap towards the major open problem of determining whether $c_5<1$. We also improve the exponential factor in the upper bound for $f_r(n)$. We show that for sufficiently large fixed value of $r$,
\[f_r(n)
\le
\left[
\left(\frac{12}{11}\right)^{3/4}
\frac{e\ln(12/11)}{4} \cdot
r
\left(\frac{11}{12}\right)^{r/4}
+o(1)
\right]
\binom{n}{\lfloor r/2 \rfloor}.
\]

Several related problems studied in the literature include multicoverings of complete graphs and hypergraphs~\cite{babu2021multicovering}, decompositions of products of complete graphs and hypergraphs~\cite{leader2017decomposing,leader2018improved}, variants of biclique partitions~\cite{ShigetaAmano2015}, odd covers~\cite{babai1988linear,radhakrishnan2000depth,buchanan2023odd,buchanan2024odd,leader2026odd}, and more general list-covering problems~\cite{cioabua2013variations,babu2022improved,alon2023new,grytczuk2026neighborly,cheng2024exact}.

\section{The Main Result}

\subsection{Decomposition of $E(K_4)\times E(K_n)$}

For an $r$-uniform hypergraph $H$, let $f_r(H)$ be the minimum number of complete $r$-partite $r$-uniform hypergraphs needed to partition the edge set of $H$. So $f_r(K_n^{r})=f_r(n)$. A \emph{decomposition} of an $r$-uniform hypergraph $H$ is a partition of the edge set of $H$ into $f_r(H)$ complete $r$-partite $r$-graphs. For graphs $G$ and $H$, let $g(G,H)$ be the minimum number of sets
of the form $E(K_{a,b})\times E(K_{c,d})$ whose disjoint union is $E(G)\times E(H)$. We alternatively represent $g(K_n, K_n)$ with $g(n)$. Taking products of the edge sets of
bicliques in the decompositions of the two complete graphs provided by Graham-Pollak yields $g(n)\le(n-1)^2$. Leader, Mili{\'c}evi{\'c}, and Tan~\cite{leader2017decomposing}
improved this bound to $g(n)\le(14/15+o(1))n^2$.
Lemma~\ref{lem:k4-k17} below gives a finite construction that,
through Proposition~\ref{prop:finite-transfer}, improves this
coefficient further to $11/12$.

We use the following construction from
\cite[Section~4]{leader2017decomposing}. Let $G_0,G_1,G_2,G_3$
be graphs on the same $n$-vertex set whose edge sets partition
$E(K_n)$. Let $M_1,M_2,M_3$ be the three perfect matchings of
$K_4$, and let $C_i=K_4-M_i$, where each $C_i$ is a $4$-cycle. The sets $E(C_1) \times E(G_1)$, $E(C_2) \times E(G_2)$, $E(C_3) \times E(G_3)$, $E(K_4-C_1) \times E(G_0 \cup G_1)$, $E(K_4-C_2) \times E(G_0 \cup G_2)$, $E(K_4-C_3) \times E(G_0 \cup G_3)$ form a partition of $E(K_4) \times E(K_n)$. Therefore,

\begin{equation}\label{eq:four-pattern}
g(K_4,K_n)\le
\sum_{i=1}^{3}\bigl(f_2(G_i)+2f_2(G_0\cup G_i)\bigr).
\end{equation}

\begin{lemma}\label{lem:k4-k17}
    The set $E(K_4) \times E(K_{17})$ can be decomposed into  $44$ sets of the form $E(K_{a,b}) \times E(K_{c,d})$.
\end{lemma}

\begin{proof}
Label the vertices of $K_{17}$ by $\{1,\ldots,9,a,b,c,d,e,f,g,h\}$, where the letters $a$ through $h$ represent $10$ through $17$. Let $G_0, G_1, G_2$ and $G_3$ be graphs that form a decomposition of $K_{17}$, defined as follows:
\begingroup
\raggedbottom
\allowdisplaybreaks[1]
\begin{align*}
E(G_0)={}&\{13,1b,27,28,2b,35,36,3a,4a,68,\\*
&\quad 6f,6g,7d,7e,9f,ac,af,be,ce,df\}\\
E(G_0 \cup G_1)={}&\{ij: i \in \{2,3\},j \in \{1,4,5,6,7,8,a,b\}\}\\*
&\cup \{ij: i \in \{6,7\},j \in \{5,8,d,g\}\}\\*
&\cup \{ij: i \in \{a,b\},j \in \{1,4,9,c\}\}\\*
&\cup \{ij: i \in \{e,f\},j \in \{6,7,9,a,b,c,d,g\}\}\\
E(G_0 \cup G_2)={}&\{ij: i \in \{1,2\},j \in \{3,7,8,9,b,h\}\}\\*
&\cup \{ij: i \in \{5,6\},j \in \{1,3,8,b,c,f,g,h\}\}\\*
&\cup \{ij: i \in \{9,a\},j \in \{3,4,5,6,7,8,c,d,f,h\}\}\\*
&\cup \{ij: i \in \{d,e\},j \in \{1,2,5,7,b,c,f,h\}\}\\
E(G_0 \cup G_3)={}&\{ij: i \in \{3,4\},j \in \{1,5,6,a,c,d,e,h\}\}\\*
&\cup \{ij: i \in \{7,8\},j \in \{2,4,6,b,c,d,e,h\}\}\\*
&\cup \{ij: i \in \{b,c\},j \in \{1,2,a,e,g,h\}\}\\*
&\cup \{ij: i \in \{f,g\},j \in \{1,2,3,4,6,8,9,a,d,h\}\}\\
E(G_1)={}&\{ij: i \in \{2\},j \in \{1,4,5,6,a\}\}\\*
&\cup \{ij: i \in \{7\},j \in \{3,5,8,f,g\}\}\\*
&\cup \{ij: i \in \{b\},j \in \{3,4,9,c,f\}\}\\*
&\cup \{ij: i \in \{e\},j \in \{6,9,a,d,g\}\}\\*
&\cup \{ij: i \in \{3\},j \in \{4,8\}\}\\*
&\cup \{ij: i \in \{6\},j \in \{5,d\}\}\\*
&\cup \{ij: i \in \{a\},j \in \{1,9\}\}\\*
&\cup \{ij: i \in \{f\},j \in \{c,g\}\}\\
E(G_2)={}&\{ij: i \in \{2\},j \in \{3,d,h\}\}\\*
&\cup \{ij: i \in \{5\},j \in \{8,d,f,g\}\}\\*
&\cup \{ij: i \in \{1,9,a\},j \in \{5,6,7,8,d,h\}\}\\*
&\cup \{ij: i \in \{1,2,3,4,c\},j \in \{9\}\}\\*
&\cup \{ij: i \in \{5,6,d\},j \in \{b,c,h\}\}\\*
&\cup \{ij: i \in \{1,2,5,f,h\},j \in \{e\}\}\\
E(G_3)={}&\{ij: i \in \{4,7\},j \in \{6,c,h\}\}\\*
&\cup \{ij: i \in \{3,8\},j \in \{d,e,g,h\}\}\\*
&\cup \{ij: i \in \{1,5,7,8,d,e,f,g\},j \in \{4\}\}\\*
&\cup \{ij: i \in \{1,2,3,8,h\},j \in \{c,f\}\}\\*
&\cup \{ij: i \in \{1,2,9,a,c,d,h\},j \in \{g\}\}\\*
&\cup \{ij: i \in \{7,8,a,g,h\},j \in \{b\}\}
\end{align*}
\endgroup

This decompositions give $f_2(G_0\cup G_i)\le4$ for $i=1,2,3$, $f_2(G_1)\le8, f_2(G_2)\le6, f_2(G_3)\le6$.
Applying \eqref{eq:four-pattern} yields
\[
\begin{aligned}
g(K_4,K_{17})
&\le\sum_{i=1}^{3}\bigl(f_2(G_i)+2f_2(G_0\cup G_i)\bigr)\\
&\le(8+6+6)+2(4+4+4)=44.
\end{aligned}
\]
\end{proof}

A complete list of the 44 product blocks is provided in Appendix~\ref{app:product-blocks}, together with a code that independently verifies that they partition $E(K_4)\times E(K_{17})$.

\begin{proposition}[\cite{leader2017decomposing}, Proposition~1]
\label{prop:transfer-f4}
Let $\alpha,C>0$ be constants. If $g(n)\le\alpha n^2+Cn$ for
every positive integer $n$, then
\[
f_4(n)\le\alpha(1+o(1))\frac{n^2}{2}
\]
for sufficiently large $n$.
\end{proposition}

\begin{proposition}[\cite{leader2017decomposing}, Proposition~2]
\label{prop:finite-transfer}
Suppose $g(K_a,K_b)<(a-1)(b-1)$ for some integers $a,b\ge2$.
Then there is a constant $C>0$ such that, for every positive
integer $n$,
\[
g(n)\le\alpha n^2+Cn,\qquad
\alpha=\frac{g(K_a,K_b)}{(a-1)(b-1)}.
\]
\end{proposition}

Combining Lemma~\ref{lem:k4-k17} with
Proposition~\ref{prop:finite-transfer} gives
$g(n)\le(11/12)n^2+Cn$ for some constant $C>0$, since
$44/((4-1)(17-1))=11/12$.
Applying Proposition~\ref{prop:transfer-f4} with $\alpha=11/12$
then yields the following corollary.

\begin{corollary}\label{cor:four-uniform}
For sufficiently large $n$,
\[
f_4(n)\le\left(\frac{11}{12}+o(1)\right)\binom{n}{2}.
\]
\end{corollary}

Corollary~\ref{cor:four-uniform} improves the coefficient $14/15$
in \cite[Theorem~11]{leader2017decomposing} to $11/12$.

\subsection{Improved bounds for $f_r(n)$}

Consider the odd-uniform case, where the existing constructions leave a larger leading constant. In the case of hypergraphs, Leader and
Tan~\cite{leader2018improved} partitioned the vertex set and
decomposed most edges by pairing vertex classes and applying improved
decompositions of products of complete graphs. For $r=2d+1$,
their approach leaves two families of edge classes corresponding
to the intersection patterns $1+2+\cdots+2$ and $3+2+\cdots+2$,
which together determine the leading term in the upper bound.

Our approach treats these remaining edge classes collectively rather than individually. We partition $[d]$ into subsets $P_1,\ldots,P_\ell$ and group the edge classes according to the
position of the part of size $3$. We then construct an exact cover
for each group. Using $g(n)\le(11/12)n^2+O(n)$ from
Lemma~\ref{lem:k4-k17} and Proposition~\ref{prop:finite-transfer},
this reduces the leading contribution to
\[
\sum_{j=1}^{\ell}\left(\frac{11}{12}\right)^{\lfloor(d-|P_j|)/2\rfloor},
\]
thereby improving the upper bound for $c_r$. In particular, this
construction implies that $c_r<1$ for every odd $r\ge65$.

Let $V_1,\ldots,V_k$ be pairwise disjoint sets and let $a_1,\ldots,a_k$ be non-negative integers. Define $ \displaystyle \binom{V_1}{a_1} \times \cdots \times \binom{V_k}{a_k}$ as the set of all subsets $X$ of $V_1 \cup \cdots \cup V_k$ such that $|X \cap V_i|=a_i$ for every $1 \leq i \leq k$. We also denote all such subsets as $\displaystyle \prod_{i=1}^k \binom{V_i}{a_i}$. A set $\Gamma$ of complete $r$-partite $r$-graphs over $V_1 \cup \cdots \cup V_k$ is said to {\em exactly cover} a hypergraph $F$, if the hypergraphs in $\Gamma$ are edge-disjoint and the union of the edges of the hypergraphs in $\Gamma$ is $F$. Recall that a complete $r$-partite $r$-graph is also referred to as a {\em block}. So $f_r(n)$ is the minimum number of complete $r$-partite $r$-graphs required to exactly cover the edge set of the complete $r$-uniform hypergraph on $n$ vertices. Note that the aforementioned result of Leader, Mili{'c}evi{'c}, and Tan~\cite{leader2017decomposing} translates to, for two disjoint sets $U$ and $V$ each of size $n$, the number of blocks with four parts needed to exactly cover $\displaystyle \binom{U}{2} \times \binom{V}{2}$ is $g(n)$.

Let $H_i$ be a complete $r_i$-partite $r_i$-graph with vertex classes
$A_{i,1},\ldots,A_{i,r_i}$, for $1\le i\le t$, and suppose that
the vertex sets of $H_1,\ldots,H_t$ are pairwise disjoint. Define the
\emph{product} of $H_1,\ldots,H_t$ to be the complete
$(r_1+\cdots+r_t)$-partite $(r_1+\cdots+r_t)$-graph $H_1\times\cdots\times H_t$ whose vertex classes are $A_{1,1},\ldots, \allowbreak A_{1,r_1}, A_{2,1},\ldots, A_{2,r_2}, \ldots, A_{t,1},\ldots, A_{t,r_t}$. Let $\mathcal{B}_1,\ldots,\mathcal{B}_t$ be families of complete
multipartite hypergraphs on pairwise disjoint vertex sets. Define the
\emph{product} of these families by $
\mathcal{B}_1\times\cdots\times\mathcal{B}_t
=
\left\{
H_1\times\cdots\times H_t:
H_i\in\mathcal{B}_i
\text{ for }1\le i\le t
\right\}$.
Throughout this section, fix an ordering of the elements of each set
$V_i$. For each $x\in V_i$, let $V_i^{<x}$ and $V_i^{>x}$ denote the
sets of elements of $V_i$ that precede and follow $x$, respectively, in
this ordering.
\begin{lemma}\label{lemma1}
Let $V_1,\ldots,V_s$ be pairwise disjoint sets, each of size $n$. Then
\[\displaystyle
\bigcup_{t \in [s]}
\Biggl[
\binom{V_t}{3}
\prod_{\substack{i \in [s]\\ i \neq t}}
\binom{V_i}{2}
\Biggr]\]
admits an exact cover by at most $n^s$ blocks.
\end{lemma}

\begin{proof}
For each choice of $\{x_1,\ldots,x_s\}$, where $x_i\in V_i$ for $1\leq i\leq s$, we define a block as follows:
\begin{equation*}
\begin{split}
&\{V_1^{< x_1}\},\{x_1\},\cdots ,\{V_s^{< x_s}\},\{x_s\},\{\bigcup_{i=1}^s V_i^{> x_i}\}\\
\end{split}
\end{equation*}
Keep only those blocks in which all $2s+1$ parts are nonempty. Note that these blocks form an exact cover of the family. Since there are at most $n^s$ choices of $\{x_1,\ldots,x_s\}$, the result follows.
\end{proof}

\begin{lemma}\label{lemma2}
Let $V_1,\ldots,V_d$ be pairwise disjoint sets, each of size $n$, and let $P\subseteq[d]$ with $|P|=s$. Then
$\displaystyle \bigcup_{t\in P}
\Biggl[
\binom{V_t}{3}
\prod_{\substack{i\in[d]\\ i\neq t}}
\binom{V_i}{2}
\Biggr]$
admits an exact cover with at most
\[
\begin{cases}
n^{s}\,g(n)^{(d-s)/2},
& \text{if $d-s$ is even},\\[1ex]
n^{s+1}\,g(n)^{(d-s-1)/2},
& \text{if $d-s$ is odd},
\end{cases}
\]
blocks.
\end{lemma}

\begin{proof}
Let $P=\{i_1,\ldots,i_s\}$. For each choice of $\{x_{i_1},\ldots,x_{i_s}\}$, where $x_{i_k}\in V_{i_k}$ for $i_k \in P$, let $B_u$ be the block as follows:
\begin{equation*}
\begin{split}
&\{V_{i_1}^{< x_{i_1}}\},\{x_{i_1}\},\cdots ,\{V_{i_s}^{< x_{i_s}}\},\{x_{i_s}\},\{\bigcup_{k=1}^s V_{i_k}^{> x_{i_k}}\}\\
\end{split}
\end{equation*}
Keep only those blocks whose $2s+1$ parts are all nonempty. By Lemma~\ref{lemma1}, these blocks form an exact cover of $\displaystyle \bigcup_{t\in P} \Biggl[\binom{V_t}{3} \prod_{\substack{i\in P\\ i\neq t}} \binom{V_i}{2}\Biggr].$ Let $\mathcal{\overline{B}}_{u}$ be the family of blocks $B_u$.

Let $[d]\setminus P=\{j_1,\ldots,j_{d-s}\}$. Pair the sets
$V_{j_1},\ldots,V_{j_{d-s}}$, leaving one set unpaired when $d-s$ is
odd. For each pair $(V_{j_{2i-1}},V_{j_{2i}})$, let $\mathcal{B}_{i}$ be a
family of $g(n)$ blocks with four parts that exactly covers 
$\displaystyle \binom{V_{j_{2i-1}}}{2}\times\binom{V_{j_{2i}}}{2}$.
If $d-s$ is odd, let $\mathcal{B}_0$ be a family of $n-1$ blocks with two parts that exactly covers $\displaystyle \binom{V_{j_{d-s}}}{2}$.

Let $\mathcal{\overline{B}}_{w}$ be the family of blocks obtained by taking the product of all the $\mathcal{B}_i$s (and $\mathcal{B}_0$, when $d-s$ is odd). Each block in $\mathcal{\overline{B}}_{w}$ has $2d-2s$ parts, and the family $\mathcal{\overline{B}}_{w}$ exactly covers 
$\displaystyle \prod_{k=1}^{d-s}\binom{V_{j_k}}{2}$.
Moreover,
\[
|\mathcal{\overline{B}}_{w}|
\le
\begin{cases}
g(n)^{(d-s)/2}, & \text{if $d-s$ is even},\\[1ex]
(n-1)\,g(n)^{(d-s-1)/2},
& \text{if $d-s$ is odd}.
\end{cases}
\]
Consider the product $\mathcal{\overline{B}}_{u} \times \mathcal{\overline{B}}_{w}$. The resulting blocks in the family have $2d+1$ parts and admit an exact cover of $\displaystyle \bigcup_{t\in P} \Biggl[
\binom{V_t}{3} \prod_{\substack{i\in[d]\\ i\neq t}} \binom{V_i}{2}
\Biggr]$. 
\end{proof}

\begin{lemma}\label{lemma3}
Let $V_1,\ldots,V_k$ be pairwise disjoint sets, each of size $n$, and let
$P\subseteq[d]$ with $|P|=s$ and $d<k$. Define $U=\bigcup_{i\notin[d]}V_i$.
Then $\displaystyle \Biggl[\binom{U}{1} \prod_{i \in P} \binom{V_i}{2}\Biggr] \cup \Biggl\{ \bigcup_{t \in P} \Biggl[ \binom{V_t}{3} \prod_{\substack{i \in P \\ i \neq t}} \binom{V_i}{2} \Biggr] \Biggr \}$
admits an exact cover using at most $n^s$ blocks.
\end{lemma}

\begin{proof}
Let $P=\{i_1,\ldots,i_s\}$. For each choice of $\{x_{i_1},\ldots,x_{i_s}\}$, where $x_{i_k}\in V_{i_k}$ for $i_k \in P$, we define a block as follows:
\begin{equation*}
\begin{split}
&\{V_{i_1}^{< x_{i_1}}\},\{x_{i_1}\},\cdots ,\{V_{i_s}^{< x_{i_s}}\},\{x_{i_s}\},\{ U \bigcup_{k=1}^s V_{i_k}^{> x_{i_k}}\}\\
\end{split}
\end{equation*}
Keep only those blocks whose $2s+1$ parts are all nonempty. These blocks admit an exact cover of the family.

\end{proof}

\begin{lemma}\label{lemma4}
Let $V_1,\ldots,V_k$ be pairwise disjoint sets, each of size $n$, and let $P\subseteq[d]$ with $|P|=s$ and $d<k$. Define $U=\bigcup_{i\notin[d]}V_i$. Then $\displaystyle \Biggl[
\binom{U}{1} \prod_{i\in[d]}\binom{V_i}{2} \Biggr] \cup
\Biggl\{\bigcup_{t\in P} \Biggl[\binom{V_t}{3}
\prod_{\substack{i\in[d]\\ i\neq t}} \binom{V_i}{2} \Biggr] \Biggr\}$
admits an exact cover of size at most
\[
\begin{cases}
n^{s}\,g(n)^{(d-s)/2},
& \text{if $d-s$ is even},\\[1ex]
n^{s+1}\,g(n)^{(d-s-1)/2},
& \text{if $d-s$ is odd},
\end{cases}
\]
blocks.
\end{lemma}

\begin{proof}
Let $P=\{i_1,\ldots,i_s\}$. Consider the family of blocks $\mathcal{\overline{B}}_{u}$ as mentioned in Lemma~\ref{lemma3}, that exactly covers $ \displaystyle \Biggl[
\binom{U}{1}
\prod_{i\in P}\binom{V_i}{2}
\Biggr]
\cup
\Biggl\{
\bigcup_{t\in P}
\Biggl[
\binom{V_t}{3}
\prod_{\substack{i\in P\\ i\neq t}}
\binom{V_i}{2}
\Biggr]
\Biggr\}.$ Let $[d]\setminus P=\{j_1,\ldots,j_{d-s}\}$. Consider the family of blocks $\mathcal{\overline{B}}_{w}$ as mentioned in Lemma~\ref{lemma2}, that exactly covers $\displaystyle \prod_{k=1}^{d-s}\binom{V_{j_k}}{2}$. Thus the product $\mathcal{\overline{B}}_{u} \times \mathcal{\overline{B}}_{w}$, has blocks with $2d+1$ parts and admits an exact cover of
$\displaystyle \Biggl[
\binom{U}{1}
\prod_{i\in[d]}\binom{V_i}{2}
\Biggr]
\cup
\Biggl\{
\bigcup_{t\in P}
\Biggl[
\binom{V_t}{3}
\prod_{\substack{i\in[d]\\ i\neq t}}
\binom{V_i}{2}
\Biggr]
\Biggr\}.$
\end{proof}

\begin{lemma}\label{lemma5}
Let $V_1,\ldots,V_k$ be pairwise disjoint sets of size $n$, and let $P=\{P_1,\ldots,P_l\}$ be a partition of $[d]$, where $|P_j|=s_j$ for $1\leq j\leq l$ and $d<k$. Define $U=\bigcup_{i\notin[d]}V_i$.
Then $\displaystyle
\Biggl[
\binom{U}{1}
\prod_{i\in[d]}\binom{V_i}{2}
\Biggr]
\cup
\Biggl\{
\bigcup_{t\in[d]}
\Biggl[
\binom{V_t}{3}
\prod_{\substack{i\in[d]\\ i\neq t}}
\binom{V_i}{2}
\Biggr]
\Biggr\}$
admits an exact cover using at most 
\[
\sum_{\substack{1\le j\le \ell\\ d-s_j\text{ even}}}
n^{s_j}\,g(n)^{(d-s_j)/2}
+
\sum_{\substack{1\le j\le \ell\\ d-s_j\text{ odd}}}
n^{s_j+1}\,g(n)^{(d-s_j-1)/2}
\]
blocks.
\end{lemma}

\begin{proof}
For $1\le j\le \ell$, let
\[
\mathcal{F}_j=
\bigcup_{t\in P_j}
\left[
\binom{V_t}{3}
\prod_{\substack{i\in[d]\\ i\neq t}}
\binom{V_i}{2}
\right].
\]
Also, define
\[
\mathcal{F}_0=
\left[
\binom{U}{1}
\prod_{i\in[d]}\binom{V_i}{2}
\right].
\]
Then
\[
\mathcal{F}_0\cup
\bigcup_{j=1}^{\ell}\mathcal{F}_j
=
\left[
\binom{U}{1}
\prod_{i\in[d]}\binom{V_i}{2}
\right]
\cup
\left\{
\bigcup_{t\in[d]}
\left[
\binom{V_t}{3}
\prod_{\substack{i\in[d]\\ i\neq t}}
\binom{V_i}{2}
\right]
\right\}.
\]

Apply Lemma ~\ref{lemma4} with $P=P_j$. This yields an exact cover of $\mathcal{F}_0\cup\mathcal{F}_j$ using at most $n^{s_j}\,g(n)^{(d-s_j)/2}$ blocks, if $d-s_j$ is even and $n^{s_j+1}\,g(n)^{(d-s_j-1)/2}$ blocks, if $d-s_j$ is odd.

The families $\mathcal{F}_1,\ldots,\mathcal{F}_{\ell}$ are pairwise
disjoint, because each edge contains three vertices in exactly one set
$V_t$, and the index $t$ belongs to a unique part of the partition
$P$. Applying Lemma~\ref{lemma2} with $P=P_i$, for each $1 \leq i \neq j \leq \ell$ yields an exact cover for $\mathcal{F}_i$ using $n^{s_i}\,g(n)^{(d-s_i)/2}$ blocks, if $d-s_i$ is even and $n^{s_i+1}\,g(n)^{(d-s_i-1)/2}$ blocks, if $d-s_i$ is odd.

Combining these covers yields an exact cover of $\mathcal{F}_0\cup
\bigcup_{j=1}^{\ell}\mathcal{F}_j$. 
\end{proof}

\begingroup
\let\originaltrivlist\trivlist
\renewcommand{\trivlist}{\setlength{\topsep}{6pt}\originaltrivlist}
\begin{theorem}\label{maintheorem}
\setlength{\abovedisplayskip}{6pt}
\setlength{\belowdisplayskip}{6pt}
\setlength{\abovedisplayshortskip}{6pt}
\setlength{\belowdisplayshortskip}{6pt}
Let $r=2d+1$ be fixed. Then for each $k$, there exists $\varepsilon_k\to0$ as
$k\to\infty$, such that for all $n$
\[
f_r(kn)
\le
\min_{s_1+ \cdots +s_{\ell}=d} \left[\sum_{j=1}^{\ell}\left(\frac{11}{12}\right)^{\lfloor (d-s_j)/2\rfloor} + \varepsilon_k \right](1+o(1))\binom{kn}{d},
\]
where $P= P_1\cup\ldots\cup P_\ell$ is a partition of $[d]$ with
$|P_j|=s_j$. (Here $o(1)$ term is as $n\to\infty$, with $k$ and $d$ fixed.)
\end{theorem}

\begin{proof}
Partition the vertex set of $K_{kn}^{(r)}$ into $k$ classes
$V_1,\ldots,V_k$, each of size $n$. Classify the $r$-edges according
to their intersection sizes with these vertex classes. For each partition of $
r=r_1+r_2+\cdots+r_\ell$, where $r_1\le r_2\le\cdots\le r_\ell$, and for each choice of distinct vertex classes $V_{i_1},\ldots,V_{i_\ell}$,
consider the family of $r$-edges $e$ satisfying $|e\cap V_{i_j}|=r_j$, for $1\le j\le\ell$.

It follows from the proof of Theorem~1 of
Leader and Tan \cite{leader2018improved} that the set of $r$-edges $e$ for which $|e\cap V_i|$ is odd for at least three distinct vertex classes $V_i$ can be decomposed into $C n^{d-1}$
complete $r$-partite $r$-graphs, and the set of $r$-edges $e$ such that $e$ intersects at most $d-1$ vertex classes and $|e\cap V_i|$ is odd for exactly one $V_i$ can be decomposed into at most $C' k^{d-1} n^d$ complete $r$-partite $r$-graphs for some constants $C,C'$ depending on $d$ and $k$.

Hence, it remains to consider only the edge classes arising from the
partitions $r=1+2+\cdots+2$ and $r=2+2+\cdots+2+3$. Applying Lemma~\ref{lemma5}, for each of $\binom{k}{d}$ collections of $d$ vertex classes, $\{V_{i_1},\ldots,V_{i_d}\}$ can be decomposed into $\sum_{\substack{1\le j\le \ell\\ d-s_j\text{ even}}}
n^{s_j}\,g(n)^{(d-s_j)/2}
+ \sum_{\substack{1\le j\le \ell\\ d-s_j\text{ odd}}}
n^{s_j+1}\,g(n)^{(d-s_j-1)/2}$ complete $r$-partite $r$-graphs where $\sum_{i=1}^{\ell} s_i=d$.


Combining the above with $g(n)\le(11/12)n^2+O(n)$, we have
\begin{equation*}
    \begin{split}
f_r(kn)
&\le
\binom{k}{d}
\Biggl[
\sum_{\substack{1\le j\le \ell\\ d-s_j\text{ even}}}
n^{s_j}\,g(n)^{(d-s_j)/2}
+
\sum_{\substack{1\le j\le \ell\\ d-s_j\text{ odd}}}
n^{s_j+1}\,g(n)^{(d-s_j-1)/2}
\Biggr]\\
&\qquad {}
+C' k^{d-1}n^d+Cn^{d-1}\\
&\le
\binom{k}{d}
\Biggl[
\sum_{\substack{1\le j\le \ell\\ d-s_j\text{ even}}}
\left(\frac{11}{12}\right)^{(d-s_j)/2}
+
\sum_{\substack{1\le j\le \ell\\ d-s_j\text{ odd}}}
\left(\frac{11}{12}\right)^{(d-s_j-1)/2}
\Biggr]n^d\\
&\qquad {}
+C' k^{d-1}n^d+Cn^{d-1}\\
&\le
\binom{k}{d}
\sum_{j=1}^{\ell}
\left(\frac{11}{12}\right)^{\lfloor (d-s_j)/2\rfloor}
n^d
+C' k^{d-1}n^d+o(n^{d})\\
&\le
\Biggl[
\sum_{j=1}^{\ell}
\left(\frac{11}{12}\right)^{\lfloor (d-s_j)/2\rfloor}
+\frac{d!C'}{k}
\Biggr]
\binom{k}{d}n^d+o(n^{d})\\
&\le
\Biggl[
\sum_{j=1}^{\ell}
\left(\frac{11}{12}\right)^{\lfloor (d-s_j)/2\rfloor}
+\varepsilon_k
\Biggr](1+o(1)) \binom{kn}{d}.
\end{split}
\end{equation*}
\end{proof}
\endgroup

\begin{corollary}\label{cor1old}
Let $r\ge65$ be a fixed odd integer. Then there exists $c_r<1$ such that
\[
f_r(n) \leq c_r(1+o(1)) \binom{n}{\lfloor r/2\rfloor}.
\]
\end{corollary}
\begin{proof}
Let $r=65$. Since $r=2d+1$, we have $d=32$. Let $P= P_1\cup P_2$ forms a partition of $[d]$ with $|P_1|=|P_2|=16$. Thus, we have
\begin{equation*}
        \begin{split}
            \sum_{j=1}^{l}\left(\frac{11}{12}\right)^{\lfloor (d-s_j)/2\rfloor} &= \left(\frac{11}{12}\right)^{8} + \left(\frac{11}{12}\right)^{8}\\
            &\approx 0.997 < 1
        \end{split}
    \end{equation*}
    Choosing $k$ such that $\Bigl( \sum_{j=1}^{l}\left(\frac{11}{12}\right)^{\lfloor (d-s_j)/2\rfloor} + \varepsilon_k \Bigr) < 1$, we have $f_r(kn) \leq 0.997 \cdot (1+o(1))\binom{kn}{d}$ for all $n$. Similarly for $r \geq 65$, considering the partition $P=P_1 \cup P_2$ of $[d]$ with $|P_1|=\lfloor \frac{d}{2} \rfloor$ and $|P_2|=\lceil \frac{d}{2} \rceil$, we have $c_r < 1$.
\end{proof}

\begin{corollary}
    For $f_r(n) \leq c_r(1+o(1))\binom{n}{\lfloor r/2 \rfloor}$, the constant $c_r$ satisfies
    \[
    c_r \leq
    \left[\left(\frac{12}{11}\right)^{3/4}
    \frac{e\ln(12/11)}{4}\right]r
    \left(\frac{11}{12}\right)^{r/4}+o(1)
    \]
    (Here the $o(1)$ term is as $r \to \infty$).
\end{corollary}
\begin{proof}

Let
$
F(s_1,\ldots,s_\ell)
=
\sum_{j=1}^{\ell}
\left(\frac{11}{12}\right)^{(d-s_j-1)/2},
$
where $s_j\ge 1$ for $1 \leq j \leq \ell$ and $\sum_{j=1}^{\ell} s_j=d$.
Choose a partition $P=P_1 \cup \cdots \cup P_{\ell}$ of $[d]$ with
$|P_1|=\cdots=|P_{\ell}|=d/\ell$. Since
$s_1=\cdots=s_\ell=d/\ell$, the above function becomes
$
F(\ell)
=
\ell\left(\frac{11}{12}\right)^{(d-d/\ell-1)/2}.
$
Differentiating and setting $F'(\ell)=0$, we have
\[
F'(\ell)
=
\left(\frac{11}{12}\right)^{(d-d/\ell-1)/2}
\left(1-\frac{d}{2\ell}\ln \frac{12}{11}\right)
\]
Since $\left(\frac{11}{12}\right)^{(d-d/\ell-1)/2}>0,$ equation $F'(\ell)=0$ is equivalent to
\begin{equation*}
    \begin{split}
         1-\frac{d}{2\ell}\ln \frac{12}{11}=0\\
         \ell=\frac{d}{2}\ln \frac{12}{11}
    \end{split}
\end{equation*}

Differentiating $F'(\ell)$, we have
\[
F''(\ell)
=
\left(\frac{11}{12}\right)^{(d-d/\ell-1)/2}
\frac{d^2}{4\ell^3}\left(\ln\frac{12}{11}\right)^2
\]
Since $F''(\ell)>0$, the equation $F'(\ell)=0$ provides the value of $\ell$ where $F(\ell)$ achieves the minimum. Substituting $\ell^*=\frac{d}{2}\ln\frac{12}{11}$, we have
\[
\begin{split}
F\left(\ell^*\right)
&=\frac{d}{2}\ln\frac{12}{11}
\left(\frac{11}{12}\right)^{\frac{d-1}{2}-\frac{1}{\ln(\frac{12}{11})}}\\
&=\frac{d}{2}\ln\frac{12}{11}
\left(
\frac{11}{12}
\right)^{(d-1)/2}
\left(\frac{11}{12}\right)^{-1/\ln(12/11)}
\end{split}
\]
Simplifying using $\left(\frac{11}{12}\right)^{-1/\ln(12/11)}=e$, we get
\[
F\left(\ell^* \right)
=\frac{ed}{2}\ln\frac{12}{11}
\left(\frac{11}{12}\right)^{(d-1)/2}
\]
Since $r=2d+1$, this gives
\[
F\left(\ell^* \right)
=
\left(\frac{12}{11}\right)^{3/4}
\frac{e\ln(12/11)}{4}(r-1)
\left(\frac{11}{12}\right)^{r/4}.
\]

For even $r$, add a new vertex $x$ to an $n$-vertex set and take an exact cover of $K_{n+1}^{(r+1)}$. Restrict this cover to the edges that contain $x$. After deleting $x$ from each such edge, the restricted blocks give an exact cover of $K_n^{(r)}$. Hence $f_r(n)\le f_{r+1}(n+1)$. Since $\lfloor (r+1)/2\rfloor=\lfloor r/2\rfloor$ and $\binom{n+1}{\lfloor r/2\rfloor}=(1+o(1))\binom{n}{\lfloor r/2\rfloor}$, this implies that $c_r\le c_{r+1}$.

\end{proof}

\section*{Acknowledgements}
Acknowledgement of the assistance of Gurobi ILP and OpenAI Codex for obtaining the construction in Lemma 1. The authors have checked all the proofs and take full responsibility for the correctness of the results presented here.

\appendix
\refstepcounter{section}
\section*{Appendix \thesection}\label{app:product-blocks}
\subsection{Decomposition of $E(K_4)\times E(K_{17})$ into 44 product blocks}
This appendix lists explicitly the product blocks in Lemma~\ref{lem:k4-k17}.
For clarity, use disjoint vertex sets $V(K_4)=\{a,b,c,d\}$ and
$V(K_{17})=\{1,\ldots,17\}$. The letters $a,\ldots,h$ used for
vertices of $K_{17}$ in the proof of the lemma are thus replaced by
$10,\ldots,17$, respectively; the letters in this appendix refer only
to vertices of $K_4$. For disjoint nonempty sets $A,B$, write $K_{A,B}$ for the complete
bipartite graph with parts $A$ and $B$. Row $t$ of
Table~\ref{tab:product-blocks} specifies the block
\[
\mathcal{B}_t=E(K_{A_t,B_t})\times E(K_{C_t,D_t}).
\]
The blocks are pairwise disjoint and satisfy
\[
E(K_4)\times E(K_{17})=\bigcup_{t=1}^{44}\mathcal{B}_t.
\]
Equivalently, they partition the $\binom{4}{2}\binom{17}{2}=816$
four-vertex edges containing two vertices from each vertex set.

The Python source code is available
in our \href{https://github.com/anandbabunb1/Decomposition_product_bicliques}{\faGithub\enspace GitHub repository}.
The code independently checks that the 44 product blocks partition
$E(K_4)\times E(K_{17})$, verifying that each of the 816 edges is covered exactly once.

\begingroup
\footnotesize
\setlength{\tabcolsep}{3pt}
\renewcommand{\arraystretch}{1.15}
\begin{longtable}{@{}rllll@{}}
\caption{The parts defining all 44 product blocks.}\label{tab:product-blocks}\\
\hline
$t$ & $A_t$ & $B_t$ & $C_t$ & $D_t$ \\
\hline
\endfirsthead
\multicolumn{5}{c}{\tablename~\thetable\ (continued)}\\
\hline
$t$ & $A_t$ & $B_t$ & $C_t$ & $D_t$ \\
\hline
\endhead
\hline
\multicolumn{5}{r}{Continued on the next page}\\
\endfoot
\hline
\endlastfoot
1 & $\{a,b\}$ & $\{c,d\}$ & $\{2\}$ & $\{1,4,5,6,10\}$ \\
2 & $\{a,b\}$ & $\{c,d\}$ & $\{7\}$ & $\{3,5,8,15,16\}$ \\
3 & $\{a,b\}$ & $\{c,d\}$ & $\{11\}$ & $\{3,4,9,12,15\}$ \\
4 & $\{a,b\}$ & $\{c,d\}$ & $\{14\}$ & $\{6,9,10,13,16\}$ \\
5 & $\{a,b\}$ & $\{c,d\}$ & $\{3\}$ & $\{4,8\}$ \\
6 & $\{a,b\}$ & $\{c,d\}$ & $\{6\}$ & $\{5,13\}$ \\
7 & $\{a,b\}$ & $\{c,d\}$ & $\{10\}$ & $\{1,9\}$ \\
8 & $\{a,b\}$ & $\{c,d\}$ & $\{15\}$ & $\{12,16\}$ \\
9 & $\{a\}$ & $\{b\}$ & $\{2,3\}$ & $\{1,4,5,6,7,8,10,11\}$ \\
10 & $\{c\}$ & $\{d\}$ & $\{2,3\}$ & $\{1,4,5,6,7,8,10,11\}$ \\
11 & $\{a\}$ & $\{b\}$ & $\{6,7\}$ & $\{5,8,13,16\}$ \\
12 & $\{c\}$ & $\{d\}$ & $\{6,7\}$ & $\{5,8,13,16\}$ \\
13 & $\{a\}$ & $\{b\}$ & $\{10,11\}$ & $\{1,4,9,12\}$ \\
14 & $\{c\}$ & $\{d\}$ & $\{10,11\}$ & $\{1,4,9,12\}$ \\
15 & $\{a\}$ & $\{b\}$ & $\{14,15\}$ & $\{6,7,9,10,11,12,13,16\}$ \\
16 & $\{c\}$ & $\{d\}$ & $\{14,15\}$ & $\{6,7,9,10,11,12,13,16\}$ \\
17 & $\{a,c\}$ & $\{b,d\}$ & $\{2\}$ & $\{3,13,17\}$ \\
18 & $\{a,c\}$ & $\{b,d\}$ & $\{5\}$ & $\{8,13,15,16\}$ \\
19 & $\{a,c\}$ & $\{b,d\}$ & $\{1,9,10\}$ & $\{5,6,7,8,13,17\}$ \\
20 & $\{a,c\}$ & $\{b,d\}$ & $\{1,2,3,4,12\}$ & $\{9\}$ \\
21 & $\{a,c\}$ & $\{b,d\}$ & $\{5,6,13\}$ & $\{11,12,17\}$ \\
22 & $\{a,c\}$ & $\{b,d\}$ & $\{1,2,5,15,17\}$ & $\{14\}$ \\
23 & $\{a\}$ & $\{c\}$ & $\{1,2\}$ & $\{3,7,8,9,11,17\}$ \\
24 & $\{b\}$ & $\{d\}$ & $\{1,2\}$ & $\{3,7,8,9,11,17\}$ \\
25 & $\{a\}$ & $\{c\}$ & $\{5,6\}$ & $\{1,3,8,11,12,15,16,17\}$ \\
26 & $\{b\}$ & $\{d\}$ & $\{5,6\}$ & $\{1,3,8,11,12,15,16,17\}$ \\
27 & $\{a\}$ & $\{c\}$ & $\{9,10\}$ & $\{3,4,5,6,7,8,12,13,15,17\}$ \\
28 & $\{b\}$ & $\{d\}$ & $\{9,10\}$ & $\{3,4,5,6,7,8,12,13,15,17\}$ \\
29 & $\{a\}$ & $\{c\}$ & $\{13,14\}$ & $\{1,2,5,7,11,12,15,17\}$ \\
30 & $\{b\}$ & $\{d\}$ & $\{13,14\}$ & $\{1,2,5,7,11,12,15,17\}$ \\
31 & $\{a,d\}$ & $\{b,c\}$ & $\{4,7\}$ & $\{6,12,17\}$ \\
32 & $\{a,d\}$ & $\{b,c\}$ & $\{3,8\}$ & $\{13,14,16,17\}$ \\
33 & $\{a,d\}$ & $\{b,c\}$ & $\{1,5,7,8,13,14,15,16\}$ & $\{4\}$ \\
34 & $\{a,d\}$ & $\{b,c\}$ & $\{1,2,3,8,17\}$ & $\{12,15\}$ \\
35 & $\{a,d\}$ & $\{b,c\}$ & $\{1,2,9,10,12,13,17\}$ & $\{16\}$ \\
36 & $\{a,d\}$ & $\{b,c\}$ & $\{7,8,10,16,17\}$ & $\{11\}$ \\
37 & $\{a\}$ & $\{d\}$ & $\{3,4\}$ & $\{1,5,6,10,12,13,14,17\}$ \\
38 & $\{b\}$ & $\{c\}$ & $\{3,4\}$ & $\{1,5,6,10,12,13,14,17\}$ \\
39 & $\{a\}$ & $\{d\}$ & $\{7,8\}$ & $\{2,4,6,11,12,13,14,17\}$ \\
40 & $\{b\}$ & $\{c\}$ & $\{7,8\}$ & $\{2,4,6,11,12,13,14,17\}$ \\
41 & $\{a\}$ & $\{d\}$ & $\{11,12\}$ & $\{1,2,10,14,16,17\}$ \\
42 & $\{b\}$ & $\{c\}$ & $\{11,12\}$ & $\{1,2,10,14,16,17\}$ \\
43 & $\{a\}$ & $\{d\}$ & $\{15,16\}$ & $\{1,2,3,4,6,8,9,10,13,17\}$ \\
44 & $\{b\}$ & $\{c\}$ & $\{15,16\}$ & $\{1,2,3,4,6,8,9,10,13,17\}$ \\
\end{longtable}
\endgroup

\bibliographystyle{plainurl}
\bibliography{references}

\end{document}